\documentclass[a4paper]{article}

\usepackage{amsmath}
\usepackage{amsthm}
\usepackage{amssymb}

\usepackage{graphicx}

\usepackage{subfig}
\usepackage[labelfont=bf, justification=justified]{caption}

\usepackage{multirow}
\usepackage{rotating}

\usepackage{algpseudocode}
\usepackage{algorithm}

\usepackage{color}
\usepackage[normalem]{ulem}
\usepackage[utf8]{inputenc}
\usepackage{tikz}
\usetikzlibrary{automata}
\usetikzlibrary{arrows,shapes}

\newtheorem{theorem}{Theorem}
\newtheorem{lemma}{Lemma}
\newtheorem{corollary}{Corollary}

\allowdisplaybreaks

\newcommand{\X}{\mathcal{X}}

\newcommand{\Y}{\mathcal{Y}}

\newcommand{\cU}{\mathcal{U}}

\newcommand{\R}{\mathbb{R}}

\newcommand{\ol}{\overline{\lambda}}
\newcommand{\la}{\lambda}

\newcommand{\La}{\Lambda}
\newcommand{\Lab}{\overline{\Lambda}}

\newcommand{\BIGOP}[1]{\mathop{\mathchoice%
{\raise-0.22em\hbox{\huge $#1$}}%
{\raise-0.05em\hbox{\Large $#1$}}{\hbox{\large $#1$}}{#1}}}

\usepackage{authblk}

\begin{document}

\title{Compromise Solutions for Robust Combinatorial Optimization with Variable-Sized Uncertainty}

\author[1]{Andr\'{e} Chassein\thanks{Email: chassein@mathematik.uni-kl.de}}
\author[2]{Marc Goerigk\thanks{Corresponding author. Email: m.goerigk@lancaster.ac.uk}}

\date{}
\affil[1]{Fachbereich Mathematik, Technische Universit\"at Kaiserslautern, Germany}
\affil[2]{Department of Management Science, Lancaster University, United Kingdom}

\maketitle

\begin{abstract}
In classic robust optimization, it is assumed that a set of possible parameter realizations, the uncertainty set, is modeled in a previous step and part of the input. As recent work has shown, finding the most suitable uncertainty set is in itself already a difficult task. We consider robust problems where the uncertainty set is not completely defined. Only the shape is known, but not its size. Such a setting is known as variable-sized uncertainty.

In this work we present an approach how to find a single robust solution, that performs well on average over all possible uncertainty set sizes. We demonstrate that this approach can be solved efficiently for min-max robust optimization, but is more involved in the case of min-max regret, where positive and negative complexity results for the selection problem, the minimum spanning tree problem, and the shortest path problem are provided. We introduce an iterative solution procedure, and evaluate its performance in an experimental comparison.

\end{abstract}


\section{Introduction}

Classic optimization settings assume that the problem data are known exactly. Robust optimization, like stochastic optimization, instead assumes some degree of uncertainty in the problem formulation. Most approaches in robust optimization formalize this uncertainty by assuming that all uncertain parameters $\xi$ are described by a set of possible outcomes $\cU$, the uncertainty set.

For general overviews on robust optimization, we refer to \cite{Aissi2009,RObook,GoeSchoe13-AE,bertsimas-survey}.

While the discussion of properties of the robust problem for different types of uncertainty sets $\cU$ has always played a major role in the research community, only recently the data-driven design of useful sets $\cU$ has become a focus of research. In \cite{bertsimas2013data}, the authors discuss the design of $\cU$ taking problem tractability and probabilistic guarantees of feasibility into account. The paper \cite{bertsimas2009constructing} discusses the relationship between risk measures and uncertainty sets.

In distributionally robust optimization, one assumes that a probability distribution on the data is roughly known; however, this distribution itself is subject to an uncertainty set $\cU$ of possible outcomes (see \cite{goh2010distributionally,wiesemann2014distributionally}).

Another related approach is the globalized robust counterpart, see \cite{RObook}. The idea of this approach is that a relaxed feasibility should be maintained, even if a scenario occurs that is not specified in the uncertainty set. The larger the distance of $\xi$ to $\cU$, the further relaxed becomes the feasibility requirement of the robust solution.

\smallskip

In this work we present an alternative to constructing a specific uncertainty set $\cU$. Instead, we only assume knowledge of a nominal (undisturbed) scenario, and consider a set of possible uncertainty sets of varying size based on this scenario. That is, a decision maker does not need to determine the size of uncertainty (a task that is usually outside his expertise). Our approach constructs a solution for which the worst-case objective with respect to any possible uncertainty set performs well on average over all uncertainty sizes.

The basic idea of variable-sized uncertainty was recently introduced in \cite{variable}. There, the aim is to construct a set of robust candidate solutions that requires the decision maker to chose one that suits him best. In our setting, we consider all uncertainty sizes simultaneously, and generate a single solution as a compromise approach to the unknown uncertainty. We call this setting the \textit{compromise approach to variable-sized uncertainty}.

We focus on combinatorial optimization problems with uncertainty in the objective function, and consider both min-max and min-max regret robustness (see \cite{kasperski2016robust}).

This work is structured as follows. In Section~\ref{sec:var}, we briefly formalize the setting of variable-sized uncertainty. We then introduce our new compromise approach for min-max robustness in Section~\ref{sec:comp}, and for the more involved case of min-max regret robustness in Section~\ref{sec:comp2}. We present complexity results for the selection problem, the minimum spanning tree problem, and the shortest path problem in Section~\ref{sec:probs}. In Section~\ref{sec:exp}, we evaluate our approach in a computation experiment, before concluding this paper in Section~\ref{sec:conc}.


\section{Variable-Sized Uncertainty}\label{sec:var}

We briefly summarize the setting of \cite{variable}. Consider an uncertain combinatorial problem of the form
\[ \min\ \{ c^t x: x\in\X \} \tag{P($c$)} \]
with $\X\subseteq\{0,1\}^n$, and 
an uncertainty set $\cU(\la)\subseteq\R^n_+$ that is parameterized by some size $\la \in \La$. For example,
\begin{itemize}
\item interval-based uncertainty  $\cU(\la) = \prod_{i\in[n]} [(1-\la)\hat{c}_i,(1+\la)\hat{c}_i]$ with $\La \subseteq [0,1]$, or
\item ellipsoidal uncertainty $\cU(\la) = \{ c : c = \hat{c} + C\xi, \Vert \xi \Vert_2 \le \la \}$ with $\La \subseteq \mathbb{R}_+$.
\end{itemize}
We call $\hat{c}$ the \textit{nominal scenario}, and any $\hat{x}\in\X$ that is a minimizer of P($\hat{c}$) a \textit{nominal solution}.

In variable-sized uncertainty, we want to find a set of solutions $\mathcal{S}\subseteq \X$ that contains an optimal solution to each robust problems over all $\la$. Here, the robust problem is either given by the min-max counterpart
\[ \min_{x\in\X} \max_{c\in\cU(\la)} c^t x \]
or the min-max regret counterpart
\[ \min_{x\in\X} \max_{c\in\cU(\la)} \left( c^t x - \min_{y\in\X} c^ty \right). \]
In the case of min-max robustness, such a set can be found through methods from multi-objective optimization in $\mathcal{O}(|\mathcal{S}|\cdot T)$, where $T$ denotes the complexity of the nominal problem, for many reasonable uncertainty sets. However, $\mathcal{S}$ may be exponentially large. Furthermore, in some settings, a set of solutions that would require a decision maker to make the final choice may not be desirable, but instead a single solution that represents a good trade-off over all values for $\la$ is sought.

\section{Compromise Solutions in the Min-Max Model}\label{sec:comp}

In this paper we are interested in finding one single solution that performs well over all possible uncertainty sizes $\la\in \La$. To this end, we consider the problem
\[ \min_{x\in\X} val(x) \hspace{5mm} \text{with} \hspace{5mm} val(x) = \int_\La w(\la) \left( \max_{c\in\cU(\la)} c^tx \right) d\la \tag{C} \]
for some weight function $w: \La \to \mathbb{R}_+$, such that $val(x)$ is well-defined. We call this problem the compromise approach to variable-sized uncertainty. The weight function $w$ can be used to include decision maker preferences; e.g., it is possible to give more weight to smaller disturbances and less to large disturbances. If a probability distribution over the uncertainty size were known, it could be used to determine $w$. In the following, we consider (C) for different shapes of $\cU(\la)$.

\begin{theorem}\label{th1}
Let $\cU(\la) = \prod_{i\in[n]} [(1-\la)\hat{c}_i,(1+\la)\hat{c}_i]$ be an interval-based uncertainty set with $\la\in\La\subseteq[0,1]$. Then, a nominal solution $\hat{x}$ is an optimal solution of (C).
\end{theorem}
\begin{proof}
As $\max_{c\in\cU(\la)} c^tx = (1+\la)c^t x$, we get
\begin{align*}
val(x) &= \int_0^1 w(\la) \left( (1+\la) \hat{c}^tx \right) d\la \\
&=  \left( \int_0^1 (1+\la) w(\la) d\la\right)  \hat{c}^tx
\end{align*}
Therefore, a minimizer of the nominal problem with costs $\hat{c}$ is also a minimizer of (C).
\end{proof}

\begin{lemma}\label{lem1}
For an ellipsoidal uncertainty set $\cU(\la) = \{\hat{c} + C\xi : \Vert \xi \Vert_2 \le \la \}$ with $\la\in\mathbb{R}_+$, it holds that
\[ \max_{c\in\cU(\la)} c^tx = \hat{c}^t x + \la \Vert C^t x\Vert_2 \]
\end{lemma}
\begin{proof}
This result has been shown in \cite{ben1999robust} for $\la = 1$. The proof holds analogously.
\end{proof}

\begin{theorem}\label{th2}
Let $\cU(\la) = \{ \hat{c} + C\xi : \Vert \xi \Vert_2 \le \la \}$ be an ellipsoidal uncertainty set with $\la\in\La\subseteq\mathbb{R}_+$. Then, an optimal solution to (C) can be found by solving a single robust problem with ellipsoidal uncertainty.
\end{theorem}
\begin{proof}
Using Lemma~\ref{lem1}, we find
\begin{align*}
val(x) &= \int_\La w(\la) \Big( \hat{c}^tx + \la\Vert C^t x\Vert_2 \Big) d\la \\
&=  \left( \int_\La w(\la) d\la \right) \hat{c}^t x+ \left( \int_\La \la w(\la) d\la \right)\Vert C^t x\Vert_2 
\end{align*}
To find a minimizer of $val(x)$, we can therefore solve the robust counterpart of (P) using an uncertainty set $\cU(\la')$ with $\la' = (\int_\La \la w(\la) d\la)/(\int_\La w(\la) d\la)$.
\end{proof}
Note that $ (\int_0^1 \la w(\la) d\la)/(\int_0^1 w(\la) d\la) = \frac{1}{2}$, if $w(\la) = 1$, i.e., the compromise solution simply hedges against the average size of the uncertainty. In general, recall that this formula gives the centroid of the curve defined by $w$.

The results of Theorems~\ref{th1} and \ref{th2} show that compromise solutions are easy to compute, as the resulting problems have a simple structure. This is due to the linearity of the robust objective value in the uncertainty size $\la$. Such linearity does not exist for min-max regret, as is discussed in the following section.

\section{Compromise Solutions in the Min-Max Regret Model}\label{sec:comp2}

We now consider the compromise approach in the min-max regret setting. In classic min-max regret, one considers the problem
\[ \min_{x\in\X} \max_{c\in\cU(\la)} c^t x - opt(c) \]
with $opt(c) = \min_{y\in\X} c^ty$. In the following, we restrict the analysis to the better-researched interval uncertainty sets $\cU(\la) = \prod_{i\in[n]} [(1-\la)\hat{c}_i,(1+\la)\hat{c}_i]$. Ellipsoidal uncertainty have been introduced in min-max regret only recently (see \cite{chassein2016min}).

\noindent The compromise approach to variable-sized uncertainty becomes
\[ \min val(x)\hspace{5mm} \text{with} \hspace{5mm} val(x) = \int_\La w(\la) \left( \max_{c\in\cU(\la)} c^tx - opt(c) \right) d\la \]
To simplify the presentation, we assume $\La = [0,1]$ and $w(\la) = 1$ for all $\la\in\La$ in the following. All results can be directly extended to piecewise linear functions $w$ with polynomially many changepoints.

\subsection{Structure of the Objective Function}

We first discuss the objective function $val(x)$ for some fixed $x\in\X$.
Note that
\[ reg(x,\la) := \max_{c\in\cU(\la)} c^tx - opt(c) = \max_{y\in\X} (1+\la)\hat{c}^tx - \sum_{i\in[n]} \hat{c}_i (1-\la + 2\la x_i) y_i. \]
Hence, $reg(x,\la)$ is a piecewise linear function in $\la$, where every possible regret solution $y$ defines an affine linear regret function $c^t(x,\la)(x-y)$, with
\[ c_i(x,\la) = \begin{cases} 
(1+\la)\hat{c}_i & \text{ if } x_i = 1 \\
(1-\la)\hat{c}_i & \text{ if } x_i = 0
\end{cases}.\]

\begin{figure}[htbp]
\begin{center}
\includegraphics[width=0.6\textwidth]{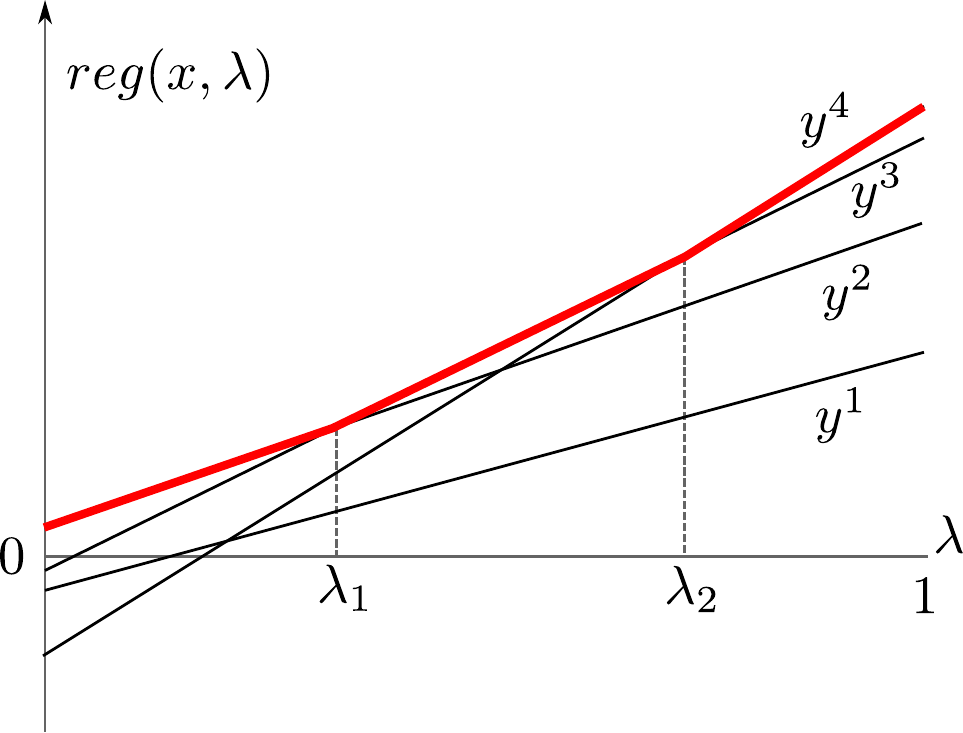}
\caption{Illustration of structure of $val(x)$.}\label{fig1}
\end{center}
\end{figure}

Figure~\ref{fig1} illustrates the objective function. In red is the maximum over all regret functions, which defines $val(x)$. On the interval $[0,\la_1]$, the regret of some solution $x$ is defined through $y^2$, while solution $y^3$ defines the regret on $[\la_1,\la_2]$, and $y^4$ defines the regret on $[\la_2,1]$. In this case, we can hence compute
\begin{align*}
val(x) &= \int_0^1 reg(x,\la) d\la = 
\int_0^1 \left( \max_{c\in\cU(\la)} c^tx - \min_{y\in\X} c^t y \right) d\la \\
&= \int_0^{\la_1} \max_{c\in\cU(\la)} c^tx - c^t y^2 d\la + \int_{\la_1}^{\la_2} \max_{c\in\cU(\la)} c^tx - c^t y^3 d\la \\
& \hspace{1cm} + \int_{\la_2}^1 \max_{c\in\cU(\la)} c^tx - c^t y^4 d\la \\
&=  \int_0^{\la_1} c^t(x,\la) (x- y^2) d\la + \int_{\la_1}^{\la_2} c^t(x,\la)(x -  y^3) d\la + \int_{\la_2}^1 c^t(x,\la)(x - y^4) d\la \\
&= \la_1 c^t(x,\frac{\la_1}{2}) (x-y^2) + (\la_2-\la_1) c^t(x,\frac{\la_1+\la_2}{2}) (x-y^3) \\
&\hspace{1cm} + (1-\la_2) c^t(x,\frac{\la_2+1}{2}) (x-y^4)
\end{align*}
In general, to compute $val(x)$, we need to determine all relevant regret solutions $y$, and the intersections of the resulting regret functions.

\subsection{Problem Formulation}

Let $\Lab(x)\subseteq\La$ be the set of changepoints of the piecewise linear function $reg(x,\cdot)$. To formulate problem (C) as a linear integer program, we use a set $\Lab \supseteq \cup_{x\in\X} \Lab(x)$. As $\X$ is a finite set, there always exists a set $\Lab$ that is finite. In general, it may contain exponentially many elements.

For the ease of notation, we assume $\Lab=\{\la_1,\ldots,\la_K\}$ with $\la_i \le \la_{i+1}$ and $\la_{K+1} := 1$. As a first approach, we model $val(x)$ by using all possible regret solutions $y\in\X$.
\begin{align}
\min\ & \sum_{\la_j \in\Lab} (\la_{j+1} - \la_j) z_j \label{p1}\\
\text{s.t. } & z_j \ge \sum_{i\in[n]} (1+\ol_j)\hat{c}_i x_i - \sum_{i\in[n]} (1-\ol_j + 2\ol_jx_i)\hat{c}_i y^\ell_i & \forall \lambda_j \in\Lab, y^\ell\in\X \nonumber\\
& x\in\X \nonumber
\end{align}
where $\overline{\la}_{j} = \frac{1}{2}(\la_j + \la_{j+1})$. 

If (P) has a duality gap of zero, i.e., if solving the dual of the linear relaxation also gives an optimal solution to (P), this formulation can be simplified. Examples where this is the case include the shortest path problem, or the minimum spanning tree problem. Let us assume that the linear relaxation of $\X$ is given by 
\[ \overline{\X} =  \{x\in\R_+^n : Ax\ge b\} \]

For the regret problem $\min_{x\in\X} \{  c^t(x,\lambda)x - opt(c(x,\lambda)) \}$ we may then write the following equivalent problem (see \cite{Aissi2009}):
\[ \min \{ c^t(x,\lambda)x - b^tu : x\in\X, u\in\overline{\Y} \} \text{ with } \overline{\Y} = \{u\ge 0 : A^tu \le c(x,\lambda)\} \]
Using this reformulation, we find the following program for (C)
\begin{align}
\min\ &  \sum_{\la_j \in\Lab} (\la_{j+1} - \la_j) \Big( c^t(x,\ol_j)x - b^t u^j \Big) \label{p2}\\
\text{s.t. } & A^tu^j \le c(x,\ol_j) & \forall \lambda_j\in\Lab \nonumber\\
& x\in\X \nonumber\\
& u^j \in \overline{\Y} & \forall \la_j\in\Lab \nonumber
\end{align}
For binary variables $x$, the product $c^t(x,\ol_j)x$ can then be linearized.
If a set $\Lab$ can be found that is of polynomial size, this is a compact formulation. In general, constraints and variables can be added in an iterative algorithm that generates new candidate values for $\lambda$ in Problem~\eqref{p2}. If the zero duality gap assumption does not hold, we can use Formulation~\eqref{p1}, where both values for $\la$ and regret solutions $y$ need to be generated. This approach is explained in the following section.

\subsection{General Algorithm}\label{sec:algo}


In the following we describe how to compute the set $\Lab(x)$ of changepoints of $reg(x,\cdot)$.
This is then used to solve Formulation~\eqref{p2} in the case of a zero duality gap for (P) as described in Algorithm~\ref{alg:c}.

\begin{algorithm}[htb] 
\caption{Exact algorithm for (C)} \label{alg:c}
\begin{algorithmic}[1]
\Require{An instance of (C).}
\State $\Lab \leftarrow \{\frac{1}{2}\}$, $k\leftarrow 0$ \label{a1s1}
\State Solve Formulation~\eqref{p2} using $\Lab$. Let the solution be $x^k$, and the objective value $LB^k$.\label{algret}
\State Compute $val(x^k)$. Let the resulting changepoints be $\Lab(x^k)$, and the objective value $UB^k$. \label{a1s3}
\If{$UB^k = LB^k$}
\State \textbf{END}: $x^k$ is an optimal solution. \label{a1s5}
\Else
\State $\Lab \leftarrow \Lab \cup \Lab(x^k)$ \label{a1s7}
\State $k \leftarrow k+1$
\State Go to \ref{algret}
\EndIf
\end{algorithmic}
\end{algorithm}
The algorithm begins with a starting set $\Lab = \{\frac{1}{2}\}$ as a guess for relevant changepoints (any other set could be used here). Using the current set $\Lab$, it then solves Formulation~\eqref{p2}. As not all constraints of the problem are present, this is a problem relaxation. Accordingly, the objective value that is found is only a lower bound $LB^k$ on the true optimal objective value of the problem. To evaluate the resulting candidate solution $x^k$, we compute $val(x^k)$ in Step~\ref{a1s3}. The sub-algorithm for this purpose is explained below. As $x^k$ is a feasible solution to (C), $val(x^k)$ gives an upper bound $UB^k$ on the optimal objective value. Hence, if lower and upper bound coincide, an optimal solution has been found. Otherwise, we extend the set $\Lab$ in Step~\ref{a1s7} and repeat the procedure.

If (P) has a duality gap larger than zero, the same algorithm can be used with the slight adjustment that Problem~\eqref{p1} is solved in Step~\ref{algret}. To this end, also regret solutions $\Y(x^k)$ generated in the computation of $val(x^k)$ need to be collected.

We describe the procedure to compute $val(x)$ in Algorithm~\ref{alg:val}.
\begin{algorithm}[htb] 
\caption{Algorithm to compute $val(x)$} \label{alg:val}
\begin{algorithmic}[1]
\Require{An instance of (C), a fixed solution $x\in\X$.}
\State $\Lab(x) \leftarrow \{0,1\}$, $\Lab^{\text{new}}(x) \leftarrow \Lab(x)$
\State $\Y(x) \leftarrow \emptyset$
\ForAll{$\la \in \Lab^{\text{new}}(x)$} \label{j1}
\State Solve (P) with costs $c(x,\la)$. Let $y$ be the resulting solution.
\State $\Y(x) \leftarrow \Y(x) \cup \{y\}$
\State $\Lab^{\text{new}}(x) \leftarrow \Lab^{\text{new}}\setminus\{\la\}$
\EndFor
\State change $\leftarrow$ false
\ForAll{$y^i,y^j\in\Y(x)$}
\State Calculate $\lambda$ as the point where the affine linear regret functions defined by $y^i$ and $y^j$ intersect.
\If{$\lambda\not\in\Lab(x)$}
\State $\Lab(x) \leftarrow \Lab(x) \cup \{\la\}$
\State $\Lab^{\text{new}}(x) \leftarrow \Lab^{\text{new}}(x) \cup \{\la\}$
\State change $\leftarrow$ true
\EndIf
\EndFor
\If{change $=$ true}
\State Go to \ref{j1}
\EndIf
\State Reduce $\Lab(x)$ and $\Y(x)$ such that only changepoints and regret functions remain that define the maximum over all affine linear functions.
\State \Return $\Lab(x)$, $\Y(x)$
\end{algorithmic}
\end{algorithm}
We begin with only two regret functions, for the extreme points $\Lab(x) = \{0,1\}$. The resulting two regret functions will intersect at one new candidate changepoint $\lambda$. We find the regret solution $y$ maximizing the regret at this point by solving a problem of type (P). We then repeat this process by iteratively calculating the regret at all current intersection points. Note that if there exists any regret function that is larger than all current regret functions at some point $\lambda$, then it is also larger than all current functions at the intersection point between two of them. Hence, Algorithm~\ref{alg:val} finds all relevant changepoints $\la$. As it may also produce unnecessary candidates $\la$, we reduce the solution sets at the end to contain only those changepoints and regret functions that define the maximum.

%

\section{Min-Max Regret Compromise Solutions for Specific Problems}\label{sec:probs}

\subsection{Minimum Selection}

The minimum selection problem is given by
\[ \min\ \left\{ c^tx : \sum_{i\in[n]} x_i = p,\ x\in\{0,1\}^n\right\} \]
and has been frequently studied in the literature on min-max regret optimization (see \cite{kasperski2016robust}). The min-max regret problem can be solved in $\mathcal{O}(n \cdot\min\{n,n-p\})$ time, see \cite{conde2004improved}.
We show that also problem (C) can be solved in polynomial time.
\begin{theorem}
Let $\cU  = \prod_{i\in[n]} [(1-\la)\hat{c}_i,(1+\la)\hat{c}_i]$ for a fixed $\la$. Then $\hat{x}$ is an optimal solution to the min-max regret selection problem.
\end{theorem}
\begin{proof}
We assume that items are sorted with respect to $\hat{c}$. Let $\tilde{x}$ be an optimal solution with $\tilde{x}_i = 0$ for an item $i\le p$. We assume $i$ is the smallest such item. Then there exists some $j>p$ with $\tilde{x}_j = 1$. Consider the solution $x'$ with $x'_k = \tilde{x}_k$ for $k\neq i,j$ and $x'_i = 1$, $x'_j = 0$.

Let $\tilde{y}$ be a regret solution for $\tilde{x}$. We can assume that $\tilde{y}_i = 1$, as $(1-\la)\hat{c}_i$ must be one of the $p$ cheapest items. We can also assume $\tilde{y}_j = 0$, as $(1+\la)\hat{c}_j$ is not among the $p$ cheapest items. Let $y'$ be the regret solution for $x'$.

The solutions $\tilde{x}$ and $x'$ differ only on the two items $i$ and $j$. Hence, the following cases are possible:
\begin{itemize}
\item Case $y'_i = 1$ and $y'_j = 0$, i.e., $\tilde{y} = y'$. We have
\begin{align*}
Reg(\tilde{x}) - Reg(x') =& (1+\la)\hat{c}_j - (1+\la)\hat{c}_i - (1-\la)\hat{c}_i + (1+\la)\hat{c}_i \\
=& (1+\la)\hat{c}_j - (1-\la)\hat{c}_i \ge 0
\end{align*}
\item Case $y_i'=1$ and $y'_j = 1$, $y'_k = 0$ for some $k > i$ with $\tilde{y}_k = 1$. Note that this means $(1+\la)\hat{c}_j \ge (1-\la+2\tilde{x}_k\la)\hat{c}_k$, as otherwise, the regret solution of $\tilde{y}$ could be improved. Hence,
\begin{align*}
Reg(\tilde{x}) - Reg(x') =& (1+\la)\hat{c}_j - (1+\la)\hat{c}_i - (1-\la)\hat{c}_i + (1+\la)\hat{c}_i \\
& - (1-\la+2\la\tilde{x}_k)\hat{c}_k + (1-\la)\hat{c}_j \\
=& (1-\la)(\hat{c}_j - \hat{c}_i) + (1+\la)\hat{c}_j - (1-\la+2\la\tilde{x}_k)\hat{c}_k \ge 0
\end{align*}
\item Case $y_i = 0$ and $y'_\ell =1$ for some $\ell>i$ with $\tilde{y}_\ell = 0$. As the costs of item $i$ have increased by using solution $x'$ instead of $\tilde{x}$, the resulting two cases are analogue to the two cases above.

\end{itemize}
 Overall, solution $x'$ has regret less or equal to the regret of $\tilde{x}$.
\end{proof}

Note that this result does not hold for general interval uncertainty sets, where the problem is NP-hard. It also does not necessarily hold for other combinatorial optimization problems; e.g., a counter-example for the assignment problem can be found in \cite{variable}.

Finally, it remains to show that $val(x)$ can also be computed in polynomial time.

\begin{theorem}
For the compromise min-max regret problem of minimum selection it holds that $|\Lab(x)| \in\mathcal{O}(\min\{p,n-p\})$ for any fixed $x\in\X$, and there is a set $\Lab$ with $|\Lab|\in\mathcal{O}(n^2)$.
\end{theorem}
\begin{proof}
If $x$ is fixed, then there are $p$ items $i$ with costs $(1+\lambda)\hat{c}_i$, and $(n-p)$ items $i$ with costs $(1-\lambda)\hat{c}_i$. The regret solution is determined by the $p$ smallest items. Accordingly, when $\la$ increases, the regret solution only changes if an item $i$ with $x_i=1$, that used to be among the $p$ smallest items, moves to the $(n-p)$ largest items, and another item $j$ with $x_j=0$ becomes part of the $p$ smallest items. There are at most $\min\{p,n-p\}$ values for $\la$ where this is the case.


We define $\Lab$ to consist of all $\la\in[0,1]$ such that
\[ (1-\la) \hat{c}_i = (1+\la) \hat{c}_j \]
for some $i,j\in[n]$, as only for such values of $\la$ an optimal regret solution may change. Hence, $|\Lab|\in\mathcal{O}(n^2)$.

\end{proof}

As the size of $\Lab(x)$ is polynomially bounded, $val(x)$ can be computed in polynomial time, and we get the following conclusion.
\begin{corollary}
The compromise min-max regret problem of minimum selection can be solved in polynomial time.
\end{corollary}

\subsection{Minimum Spanning Tree}\label{sec:mst}

The min-max regret spanning tree problem in a graph $G=(V,E)$ has previously been considered, e.g., in \cite{yaman2001robust,kasperski2011approximability}. The regret of a fixed solution can be computed in polynomial time, but it is NP-hard to find an optimal solution. We now consider the compromise min-max regret counterpart (C).

Let any spanning tree $x$ be fixed. To compute $val(x)$, we begin with $\lambda = 0$ and calculate a regret spanning tree by solving a nominal problem with costs $\hat{c}$. Recall that this can be done using Kruskal's algorithm that considers edges successively according to an increasing sorting
\[ \hat{c}_{e_1} \le \ldots \le \hat{c}_{e_{|E|}} \]
with respect to costs. If $\lambda$ increases, edges that are included in $x$ have costs $(1+\lambda)\hat{c}_e$ (i.e., their costs increase) and edges not in $x$ have costs $(1-\lambda)\hat{c}_e$ (i.e., their costs decrease). Kruskal's algorithm will only find a different solution if the sorting of edges change. As there are $|V|-1$ edges with increasing costs, and $|E|-|V|+1$ edges with increasing costs, the sorting can change at most $(|V|-1)(|E|-|V|+1) = \mathcal{O}(|E|^2)$ times (note that two edges with increasing costs or two edges with decreasing costs cannot change relative positions). We have therefore shown:

\begin{theorem}
A solution to the compromise min-max regret problem of minimum spanning tree can be evaluated in polynomial time.
\end{theorem}

If the solution $x$ is not known, we can still construct a set $\Lab$ with size $\mathcal{O}(|E|^2)$ that contains all possible changepoints along the same principle. We can conclude:
\begin{theorem}
There exists a compact mixed-integer programming formulation for the compromise min-max regret problem of minimum spanning tree.
\end{theorem}

However, we show in the following that solving the compromise problem is NP-hard. To this end, we use the following result:

\begin{theorem}{\cite{averbakh2004interval}}
The min-max regret spanning tree problem is NP-hard, even if all intervals of uncertainty are equal to $[0,1]$.
\end{theorem}

\noindent Note that if all intervals are of the form $[a,b]$, then
\begin{align*} 
reg(x) &= \sum_{e\in E} b x_e - \min_{y\in\X} \left( \sum_{e\in E\atop x_e = 1} by_e + \sum_{e\in E\atop x_e = 0} ay_e\right) \\
&= (|V|-1)b - \min_{y\in\X} \left( (|V|-1)b - \sum_{e\in E\atop x_e = 0} (b-a) y_e\right) \\
&= (b-a) \max_{y\in\X} \sum_{e\in E\atop x_e = 0}  y_e
\end{align*}
Therefore, the min-max regret problem with costs $[0,1]$ is equivalent to the min-max regret problem with any other costs $[a,b]$, in the sense that objective values only differ by a constant factor and
both problems have the same set of optimal solutions. In particular, a solution $y$ that maximizes the regret of $x$ with respect to cost intervals $[a,b]$ is also a maximizer of the regret for any other cost intervals $[a',b']$. We can conclude:

\begin{theorem}
The compromise problem of min-max regret minimum spanning tree is NP-hard, even if $w(\la) = 1$ for all $\la\in[0,1]$.
\end{theorem}
\begin{proof}
Let an instance of the min-max regret spanning tree problem with cost intervals $[0,1]$ be given. Consider an instance of the compromise problem with $\hat{c}_e = 1$ for all $e\in E$, and $w(\la) = 1$. Then
\[ val(x) = \int_0^1 reg(x,\la) d\la = \int_0^1 \left( 2\la \max_{y\in\X} \sum_{e\in E\atop x_e = 0} y_e \right) d\la = \max_{y\in\X} \sum_{e\in E\atop x_e = 0} y_e \]
Hence, any minimizer of $val(x)$ is also an optimal solution to the min-max regret spanning tree problem.
\end{proof}

\subsection{Shortest Path}

For the shortest path problem, we consider $\X$ as the set of all simple $s-t$ paths in a graph $G=(V,E)$ (for the min-max regret problem, see, e.g., \cite{averbakh2004interval}). As for the minimum spanning tree problem, the regret of a fixed solution can be computed in polynomial time, but it is NP-hard to find an optimal solution.

\noindent For the compromise problem (C), we have:
\begin{align*}
reg(x,\la) &= \sum_{e\in E\atop x_e = 1} (1+\la)\hat{c}_e - \left( \min_{y\in\X}\  \sum_{e\in E\atop x_e = 1} (1+\la)\hat{c}_e y_e + \sum_{e\in E\atop x_e = 0} (1-\la)\hat{c}_e y_e \right) \\
&= \sum_{e\in E\atop x_e = 1} (1+\la)\hat{c}_e + \left( \min_{y\in\X}\  \la \sum_{e\in E\atop x_e = 1} 2\hat{c}_ey_e + (1-\la) \sum_{e\in E} \hat{c}_e y_e \right)
\end{align*}
We can interpret the minimization problem as a weighted sum of the bicriteria problem
\[ \min \left\{ \begin{pmatrix} \sum_{e\in E\atop x_e = 1} 2\hat{c}_ey_e \\ \sum_{e\in E} \hat{c}_e y_e \end{pmatrix} : y\in \X \right\} \]
The number of solutions we need to generate to compute $val(x)$ can therefore be bounded by the number of solutions we can find through such weighted sum computations (the set of extreme efficient solutions $\mathcal{E}$).

\begin{lemma}
For the compromise shortest path problem, it holds that $|\Lab(x)| \le |\mathcal{E}|$.
\end{lemma}

Depending on the graph $G$, the following bounds on the number of extreme efficient solutions $\mathcal{E}$ (see, e.g.,  \cite{ehrgott2006multicriteria}) can be taken from the literature \cite{carstensen1983complexity,variable}:
\begin{itemize}
\item for series-parallel graphs, $\mathcal{E} \in \mathcal{O}(|E|)$
\item for layered graphs with width $w$ and length $\ell$, $\mathcal{E} \in \mathcal{O}(2^{\log w \log (\ell+1)})$
\item for acyclic graphs, $\mathcal{E} \in \mathcal{O}(|V|^{\log|V|})$
\item for general graphs $\mathcal{E}\in 2^{\Omega(\log^2|V|)}$
\end{itemize}

We can conclude:
\begin{corollary}
A solution to the compromise min-max regret problem of shortest path can be evaluated in polynomial time on series-parallel graphs and layered graphs with fixed width or length.
\end{corollary}

Note that the number of extreme efficient solutions is only an upper bound on $\Lab(x)$. Unfortunately, we cannot hope to find a better performance than this bound, as the following result demonstrates.

\begin{theorem}
For any bicriteria shortest path instance with costs $(a,b)$, $a_e>0$ for all $e\in E$, there is an instance of (C) and a solution $x$ where $\Lab(x) = |\mathcal{E}|$.
\end{theorem}
\begin{proof}
Let an instance of the bicriteria shortest path problem be given, i.e., a directed graph $G=(V,E)$ with arc costs $c_e = (a_e,b_e)$ for all $e\in E$. As $a_e > 0$ for all $e\in E$, we can assume w.l.o.g. that $2a_e \ge b_e$ for all $e\in E$.
We create the following instance of (C). 

Every arc $e=(i,j)\in E$ is substituted by three arcs $e'=(i,i'(e))$, $e''=(i'(e),j'(e))$ and $e'''=(j'(e),j)$. We set $\hat{c}_{e'} = a_e - \frac{b_e}{2}$, $\hat{c}_{e''} = \frac{b_e}{2}$ and $\hat{c}_{e'''} = 0$ (see Figure~\ref{fignp} for an example of such a transformation). Let $E'$, $E''$ and $E'''$ contain all edges of the respective type. Additionally, we choose an arbitrary order of edges $(e_1,\ldots,e_m)$, and create arcs $E_M = \{ (s,i'(e_1)), (j'(e_1),i'(e_2)), \ldots, (j'(e_m),t)\}$. We set costs of these arcs to be a sufficiently large constant $M$. Finally, let $x$ be the path that follows all edges in $E_M$ and $E''$. Note that edges in $E'''$ can be contracted, but are shown for better readability.

\begin{figure}[htbp]
 \captionsetup[subfigure]{justification=centering}
    \centering
    \subfloat[Bicriteria shortest path instance.]{\includegraphics[width=0.6\textwidth]{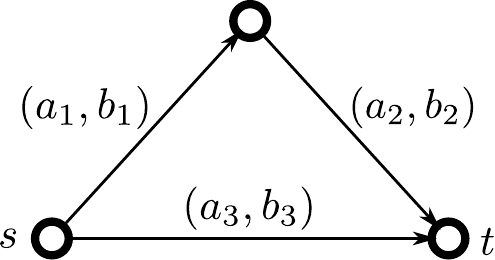}}\\
\subfloat[Compromise shortest path instance. Dashed lines indicate $x$.]{\includegraphics[width=\textwidth]{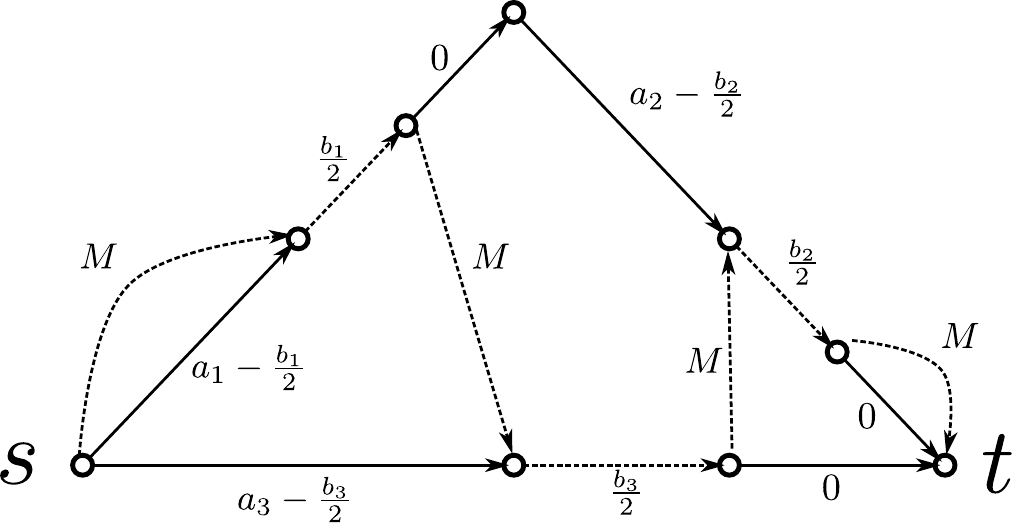}}
\caption{Example for transformation.}\label{fignp}
\end{figure}

Note that $M$ is sufficiently large so that no regret path $y$ will use any edge in $E_M$. Hence, if $y$ uses an edge in $E'$, it will also have to use the following edges in  $E''$ and $E'''$, i.e., $y$ corresponds to a path in the original graph $G=(V,E)$.
The regret of $x$ is
\begin{align*}
reg(x,\la) &= \sum_{e\in E_M} (1+\la)M + \sum_{e\in E''} (1+\la)\hat{c}_e - \min_{y\in\X} \sum_{e\in E\atop x_e = 1} (1+\la)\hat{c}_e y_e + \sum_{e\in E\atop x_e = 0} (1-\la)\hat{c}_e y_e \\
&= (1+\la)\cdot const. - \min_{y\in\X} \sum_{e\in E} \left( (1+\la)\frac{b_e}{2} + (1-\la)(a_e - \frac{b_e}{2})\right) y_e\\
&= (1+\la)\cdot const. - \min_{y\in\X} \sum_{e\in E} \left( a_e +\la (b_e - a_e) \right) y_e
\end{align*}
Therefore, if $\la$ goes from $0$ to $1$, all extreme efficient paths in the original graph $G$ are used to calculate $reg(x,\la)$. 
\end{proof}

We now consider the complexity of finding a solution $x\in\X$ that minimizes $val(x)$. Note that the reduction in \cite{averbakh2004interval} uses interval costs of the form $[0,1]$ and $[1,1]$, which does not fit into our cost framework $[(1-\la)\hat{c}_e,(1+\la)\hat{c}_e]$. Instead, we make use of the following result:

\begin{theorem}{\cite{variable}}
The min-max regret shortest path problem is NP-hard for layered graphs with interval costs $[0,1]$.
\end{theorem}
Note that for layered graphs, all paths have the same cardinality. Hence, $reg(x) = (b-a) \max_{y\in\X} \sum_{e\in E \atop x_e = 0} y_e$ (see Section~\ref{sec:mst}), and the problem with costs $[0,1]$ is equivalent to the problem with costs $[a,b]$ for any $a<b$. Analog to the last section, we can therefore conclude:

\begin{theorem}
Finding an optimal solution to the compromise shortest path problem is NP-hard on layered graphs, even if $w(\la)=1$ for all $\la\in[0,1]$.
\end{theorem}

\section{Experiments}\label{sec:exp}

\subsection{Setup}

In this section we present two experiments on compromise solutions to min-max regret problems with variable-sized uncertainty. The first experiment is concerned with the computational effort to find such a solution using the iterative algorithms presented in Section~\ref{sec:algo}. In the second experiment we compare these solutions to the alternatives of using classic regret solutions for various uncertainty set sizes.

Both experiments are conducted on shortest path instances of two types. 

\medskip

The first type consists of complete layered graphs. We parameterize such instances by the number of layers, width and cost types. Each graph consists of a source node, a sink node, and node layers of equal width between. For $N+1$ layers of width $k$, an instance has a total of $(N+1)k+2$ nodes and $Nk^2+2k$ edges. We use $N=5$ to $N=55$ in steps of size 5, and $k=5,10,15,20$. Graph sizes thus vary from 32 nodes and 130 edges to 1,122 nodes and 22,040 edges.

Edges connect all nodes of one layer to the nodes of the next layer. Source and sink are completely connected to the first and last layer, respectively. We considered two types of cost structures to generate $\hat{c}$. For type A, all costs are chosen uniformly from the interval $[1,100]$. For type B costs, we generate nominal costs in $[1,30]\cup[70,100]$, i.e., they are either low or high.

In total, there are $11\cdot 4 \cdot 2 = 88$ parameter combinations. For each combination, we generate 20 instances, i.e., a total of 1,760 instances. 

\medskip

The second type consists of graphs with two paths, that are linked by diagonal edges. For some length parameter $L$, we generated two separate paths from a node $s$ to a node $t$, each with $L$ nodes between. We then generate diagonal edges in the following way. On one of the two paths, we choose the $i$th node uniformly randomly. We then connect this node with the $j$th node on the other path, where $j>i$. The $j$th node is chosen with probability $\frac{3}{4}(\frac{1}{4})^{j-i-1}$, i.e., long diagonal edges are unlikely (ensuring that $j$ is at most $L$).

Edges along the two base paths have length chosen uniformly from the interval $[1,100]$. For diagonal edges, we determine their length by sampling from the same interval $(j-i)$ times, and adding these values, i.e.,  all paths have the same expected length.

We generate instances with length $L$ from 50 to 850 in steps of 100, and set the number of diagonal edges to be $d\cdot L$ for $d\in\{0.05,0.10,0.15\}$. The smallest instances therefore contain 102 nodes and 105 edges; the largest instances contain 1,702 nodes and  1,830 edges. For each parameter combination, we generate 20 instances, i.e., a total of $9\cdot 3 \cdot 20 = 540$ instances.

\medskip

The classic min-max regret shortest path problem on instances of both types is known to be NP-hard, see \cite{andre-diss}. We investigate both types, as we expect the nominal solution to show a different performance: For layered graphs, the nominal solution is also optimal for $\cU(1)$, as for every path there also exists a disjoint path. Therefore, the regret of a path $P$ with respect to  $\cU(1)$ is $\sum_{e\in P} 2\hat{c}_e$. For the second type of instances, a good solution with respect to min-max regret can be expected to intersect with as many other paths as possible. We can therefore expect the nominal solution to be different to the optimal solution of $\cU(1)$.

All experiments were conducted using one core of a computer with an Intel Xeon E5-2670 processor, running at 2.60 GHz with 20MB cache, with Ubuntu 12.04 and Cplex v.12.6.

\subsection{Experiment 1: Computational Effort}

\subsubsection{Layered Graphs}

We solve the compromise approach to variable-sized uncertainty for each instance using the algorithms described in Section~\ref{sec:algo} and record the computation times. Average computation times in seconds are presented in Table~\ref{tab1}. In each column we average over all instances for which this parameter is fixed; i.e., in column ''width 5'' we show the results over all 440 instances that have a width of 5, broken down into classes of different length. The results indicate that computation times are still reasonable given the complexity of the problem, and mostly depend on the size of the instance (width parameter) and the density of the graph, while the cost structure has no significant impact on computation times.

\begin{table}[htbp]
\begin{center}
\begin{tabular}{rr|rrrr|rr}
&  & \multicolumn{4}{|c|}{Width} & \multicolumn{2}{|c}{Costs} \\
&  & 5 & 10 & 15 & 20 & A & B \\
\hline
\parbox[t]{2mm}{\multirow{11}{*}{\rotatebox[origin=c]{90}{Layers}}} & 5 & 0.05 & 0.13 & 0.21 & 0.40 & 0.20 & 0.20 \\
& 10 & 0.22 & 0.49 & 0.89 & 1.43 & 0.75 & 0.77 \\
& 15 & 0.47 & 0.99 & 1.78 & 2.98 & 1.47 & 1.64 \\
& 20 & 1.17 & 2.31 & 3.61 & 6.78 & 3.45 & 3.49 \\
& 25 & 1.99 & 4.17 & 7.53 & 11.47 & 5.97 & 6.61 \\
& 30 & 3.77 & 7.86 & 13.13 & 21.14 & 11.50 & 11.45 \\
& 35 & 6.05 & 11.08 & 19.87 & 35.51 & 18.28 & 17.97 \\
& 40 & 9.46 & 21.85 & 35.37 & 49.58 & 28.64 & 29.48 \\
& 45 & 13.77 & 29.64 & 56.30 & 85.48 & 47.23 & 45.37 \\
& 50 & 21.58 & 46.37 & 67.88 & 141.14 & 66.33 & 72.16 \\
& 55 & 26.61 & 69.56 & 125.95 & 193.30 & 105.94 & 101.76
\end{tabular}
\caption{Average computation times to solve (C) in seconds.}\label{tab1}
\end{center}
\end{table}

We present more details in Tables~\ref{tab2} and \ref{tab3}, where the number of iterations (i.e., how often was the relaxation of (C) solved in Line~\ref{algret} of Algorithm~\ref{alg:c}) and the size of $\Lab$ at the end of the algorithm are presented, respectively. 

We find that the average number of iterations is stable and small, with around two iterations on average (the maximum number of iterations is three). This value seems largely independent of the problem size. For the number of generated changepoints $|\Lab|$, however, this is different. It increases with the number of layers, but it decreases with the width of the graph. Recall that the regret of a solution $x$ is roughly determined by the number of edges a regret path $y$ has in common. With increasing width, regret paths are less likely to use the same edges, which explains why the size of $\Lab(x)$ decreases. As before, we find that the cost structure does not have a significant impact on the performance of the solution algorithm.

\begin{table}[htbp]
\begin{center}
\begin{tabular}{rr|rrrr|rr}
&  & \multicolumn{4}{|c|}{Width} & \multicolumn{2}{|c}{Costs} \\
&  & 5 & 10 & 15 & 20 & A & B \\
 \hline
\parbox[t]{2mm}{\multirow{11}{*}{\rotatebox[origin=c]{90}{Layers}}} 
& 5 & 1.98 & 1.98 & 1.73 & 1.80 & 1.89 & 1.85 \\
& 10 & 2.02 & 1.98 & 1.98 & 2.00 & 1.98 & 2.01 \\
& 15 & 2.08 & 2.02 & 1.98 & 1.98 & 2.01 & 2.01 \\
& 20 & 2.08 & 2.08 & 2.00 & 2.08 & 2.06 & 2.05 \\
& 25 & 2.00 & 2.02 & 2.02 & 2.05 & 2.04 & 2.01 \\
& 30 & 2.15 & 2.10 & 2.08 & 2.05 & 2.11 & 2.08 \\
& 35 & 2.10 & 2.02 & 2.10 & 2.05 & 2.06 & 2.08 \\
& 40 & 2.10 & 2.15 & 2.10 & 2.02 & 2.11 & 2.08 \\
& 45 & 2.12 & 2.15 & 2.12 & 2.05 & 2.10 & 2.12 \\
& 50 & 2.17 & 2.08 & 2.05 & 2.15 & 2.05 & 2.17 \\
& 55 & 2.10 & 2.10 & 2.02 & 2.12 & 2.10 & 2.08
\end{tabular}
\caption{Average numbers of iterations.}\label{tab2}
\end{center}
\end{table}

\begin{table}[htbp]
\begin{center}
\begin{tabular}{rr|rrrr|rr}
&  & \multicolumn{4}{|c|}{Width} & \multicolumn{2}{|c}{Costs} \\
&  & 5 & 10 & 15 & 20 & A & B \\
 \hline
\parbox[t]{2mm}{\multirow{11}{*}{\rotatebox[origin=c]{90}{Layers}}}
& 5 & 3.05 & 2.95 & 2.17 & 2.27 & 2.59 & 2.64 \\
& 10 & 5.05 & 3.75 & 3.33 & 3.10 & 3.71 & 3.90 \\
& 15 & 6.72 & 4.95 & 4.33 & 4.10 & 4.92 & 5.12 \\
& 20 & 9.00 & 6.53 & 5.03 & 5.20 & 6.40 & 6.47 \\
& 25 & 10.47 & 7.70 & 6.65 & 5.60 & 7.42 & 7.79 \\
& 30 & 11.68 & 9.43 & 7.28 & 6.70 & 9.00 & 8.54 \\
& 35 & 13.10 & 9.32 & 7.62 & 7.10 & 8.99 & 9.59 \\
& 40 & 14.95 & 11.05 & 9.35 & 7.78 & 10.95 & 10.61 \\
& 45 & 16.10 & 11.35 & 10.35 & 8.53 & 11.85 & 11.31 \\
& 50 & 18.73 & 12.57 & 10.05 & 9.47 & 12.56 & 12.85 \\
& 55 & 19.77 & 14.40 & 11.43 & 9.62 & 13.68 & 13.94
\end{tabular}
\caption{Average size of $\Lab$ at the end of the algorithm.}\label{tab3}
\end{center}
\end{table}

\subsubsection{Two-Path Graphs}

The Tables~\ref{tab4}, \ref{tab5} and \ref{tab6} correspond to the Tables~\ref{tab1}, \ref{tab2} and \ref{tab3} from the last experiment, respectively. Computation times are sensitive to the parameter $d$, i.e., the number of diagonal edges. For small values of $d$, the computational complexity of problem (C) scales well with the length of the graph; however, for larger values of $d$, the problem becomes intractable.

\begin{table}[htbp]
\begin{center}
\begin{tabular}{rr|rrr}
& & \multicolumn{3}{|c}{$d$} \\
 & &  0.05 & 0.10 & 0.15 \\
\hline
\parbox[t]{2mm}{\multirow{8}{*}{\rotatebox[origin=c]{90}{Length}}} & 50 & 0.04 & 0.05 & 0.08 \\
& 150 & 0.13 & 0.29 & 0.67 \\
& 250 & 0.22 & 0.79 & 2.23 \\
& 350 & 0.48 & 2.07 & 11.76 \\
& 450 & 0.78 & 5.37 & 28.22 \\
& 550 & 1.33 & 12.01 & 57.44 \\
& 650 & 2.06 & 19.65 & 165.17 \\
& 750 & 3.01 & 36.70 & 488.51 \\
& 850 & 3.84 & 73.42 & 3186.18
\end{tabular}
\caption{Average computation times to solve (C) in seconds.}\label{tab4}
\end{center}
\end{table}

While the number of iterations is relatively small overall, as in the last experiment, the size of $\Lab$ increases with $d$, which makes the master problems larger and more difficult to solve.

\begin{table}[htbp]
\begin{center}
\begin{tabular}{rr|rrr}
& & \multicolumn{3}{|c}{$d$} \\
 & &  0.05 & 0.10 & 0.15 \\
\hline
\parbox[t]{2mm}{\multirow{8}{*}{\rotatebox[origin=c]{90}{Length}}}
 & 50 & 2.00 & 2.05 & 2.00 \\
 & 150 & 2.05 & 2.15 & 2.20 \\
 & 250 & 2.05 & 2.10 & 2.30 \\
 & 350 & 2.10 & 2.20 & 2.45 \\
 & 450 & 2.10 & 2.35 & 2.45 \\
 & 550 & 2.20 & 2.30 & 2.40 \\
 & 650 & 2.15 & 2.30 & 2.55 \\
 & 750 & 2.20 & 2.35 & 2.50 \\
 & 850 & 2.05 & 2.50 & 2.35
\end{tabular}
\caption{Average numbers of iterations.}\label{tab5}
\end{center}
\end{table}

\begin{table}[htbp]
\begin{center}
\begin{tabular}{rr|rrr}
& & \multicolumn{3}{|c}{$d$} \\
 & &  0.05 & 0.10 & 0.15 \\
\hline
\parbox[t]{2mm}{\multirow{8}{*}{\rotatebox[origin=c]{90}{Length}}}
 & 50 & 2.75 & 3.55 & 3.95 \\
 & 150 & 4.20 & 6.45 & 8.90 \\
 & 250 & 5.15 & 9.15 & 13.15 \\
 & 350 & 6.95 & 12.25 & 18.05 \\
 & 450 & 8.05 & 16.25 & 21.40 \\
 & 550 & 9.95 & 18.45 & 27.80 \\
 & 650 & 11.20 & 20.20 & 30.25 \\
 & 750 & 12.35 & 22.45 & 32.70 \\
 & 850 & 13.05 & 27.00 & 39.00
\end{tabular}
\caption{Average size of $\Lab$ at the end of the algorithm.}\label{tab6}
\end{center}
\end{table}

\subsection{Experiment 2: Comparison of Solutions}

\subsubsection{Layered Graphs}
\label{sec:plots}

In our second experiment, we compare the compromise solution to the nominal solution (which is also the min-max regret solution with respect to the uncertainty sets $\cU(0)$ and $\cU(1)$), and to the min-max regret solutions with respect to $\cU(0.3)$, $\cU(0.5)$ and $\cU(0.7)$.

To compare solutions, we calculate the regret of the compromise solution for values of $\la$ in $[0,1]$. We take this regret as the baseline. For all other solutions, we also calculate the regret depending on $\la$, and compute the difference to the baseline. We then compute the average differences for fixed $\la$ over all instances of the same size. The resulting average differences are shown in Figure~\ref{fig:exp} for four instance sizes. To set the differences in perspective, the average regret ranges from $\cU(0)$ to $\cU(1)$ of the compromise solutions are shown in the captions.

\begin{figure}[htbp]
   \captionsetup[subfigure]{justification=centering}
    \centering
    \subfloat[][Layers = 20, Width = 5.\\ Value range {$[0.1, 376.4]$}.]{\includegraphics[width=.5\textwidth]{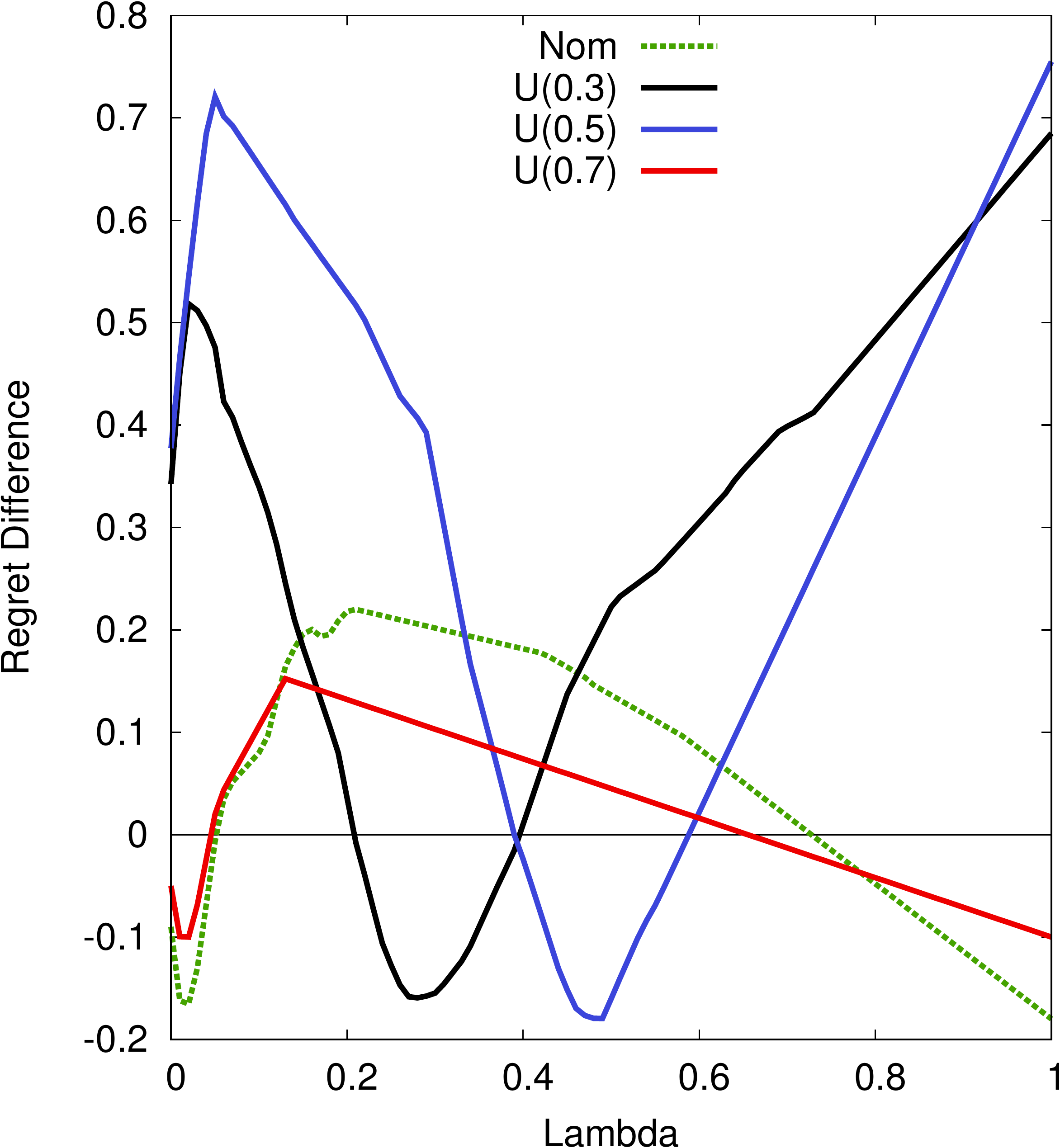}}
\subfloat[][Layers = 20, Width = 20.\\ Value range {$[0.1, 108.4]$}.]{\includegraphics[width=.5\textwidth]{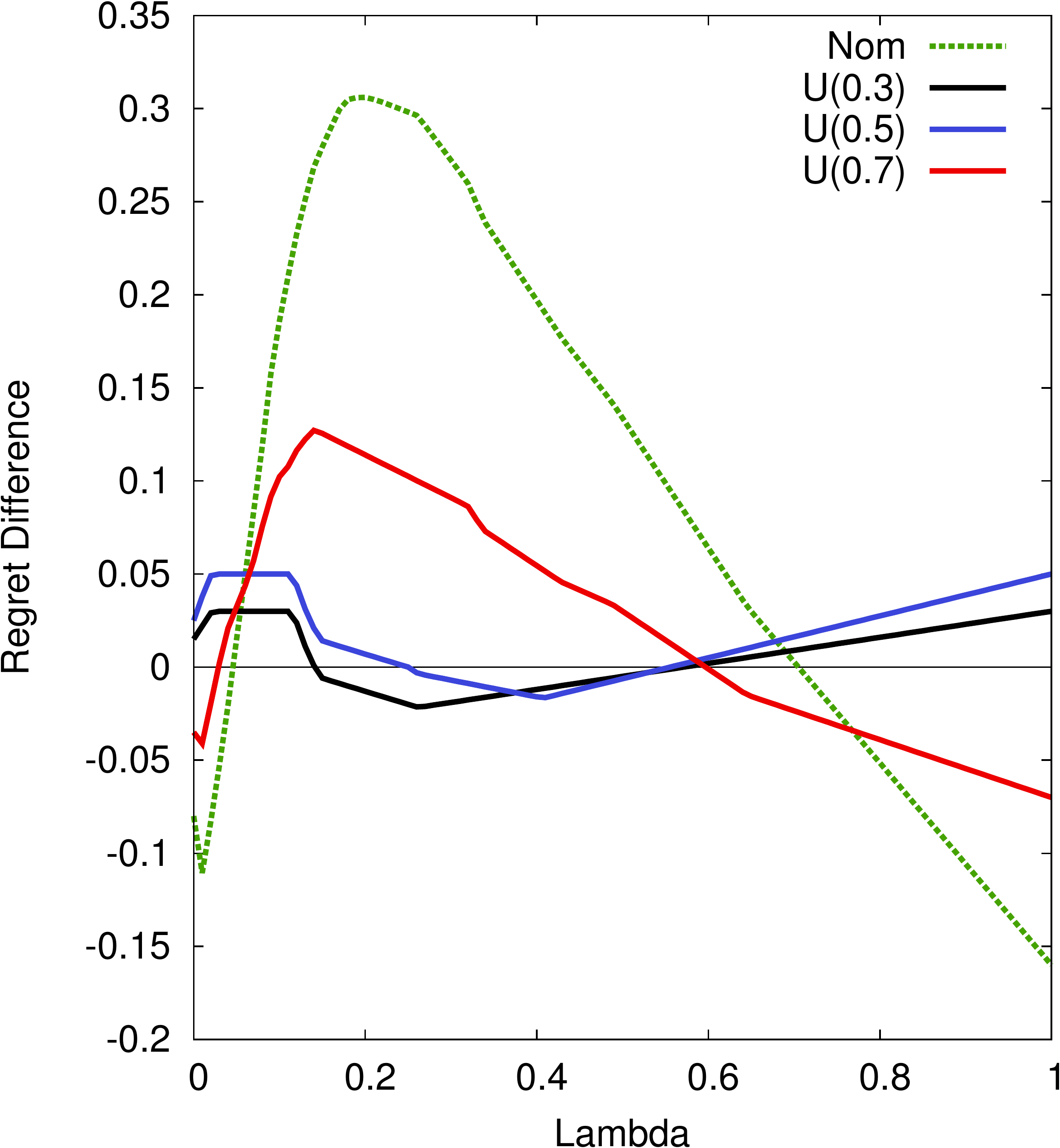}}\\
\subfloat[][Layers = 55, Width = 5.\\ Value range {$[0.5, 1055.5]$}.]{\includegraphics[width=.5\textwidth]{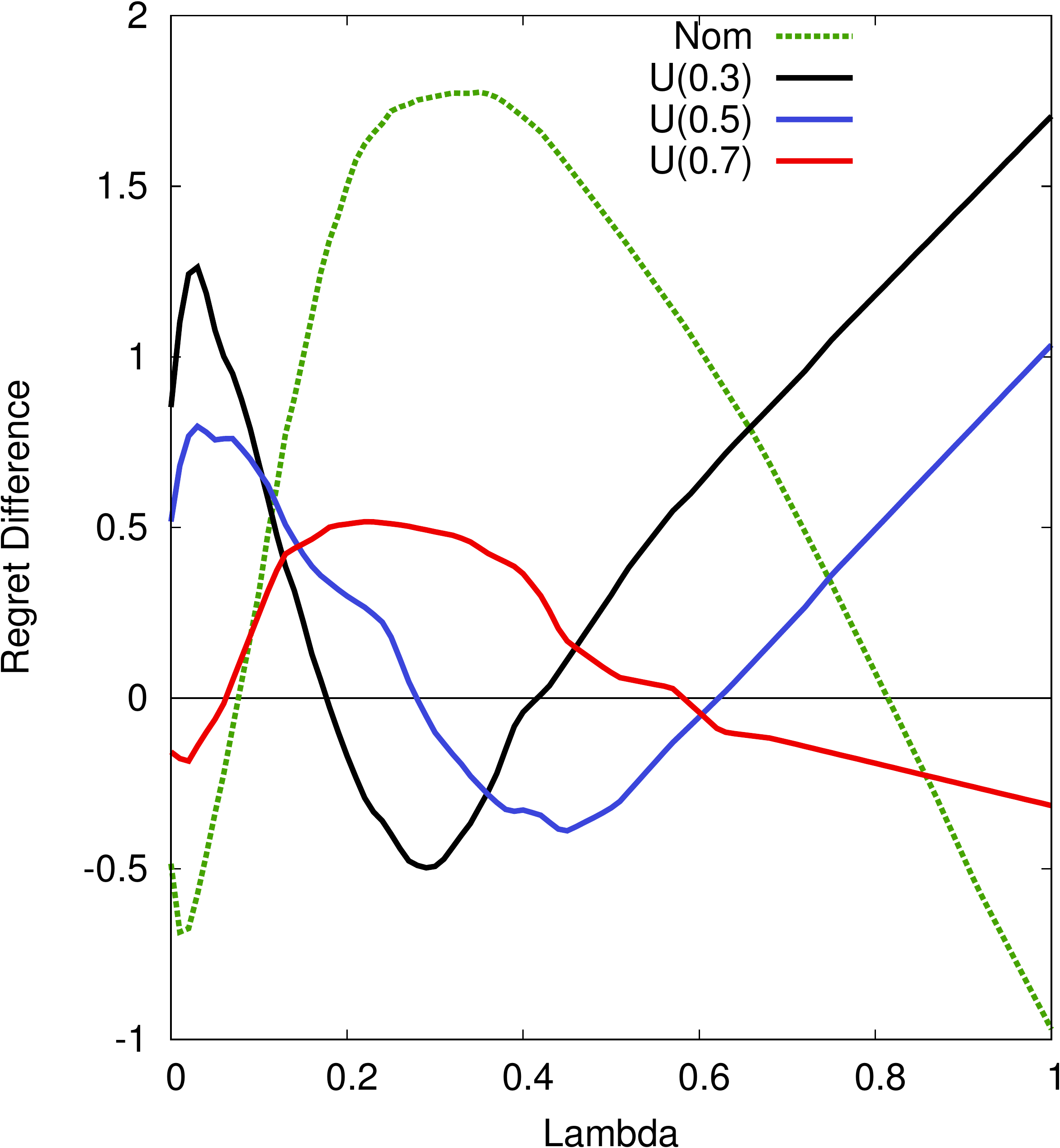}}
\subfloat[][Layers = 55, Width = 20.\\ Value range {$[0.1, 309.3]$}.]{\includegraphics[width=.5\textwidth]{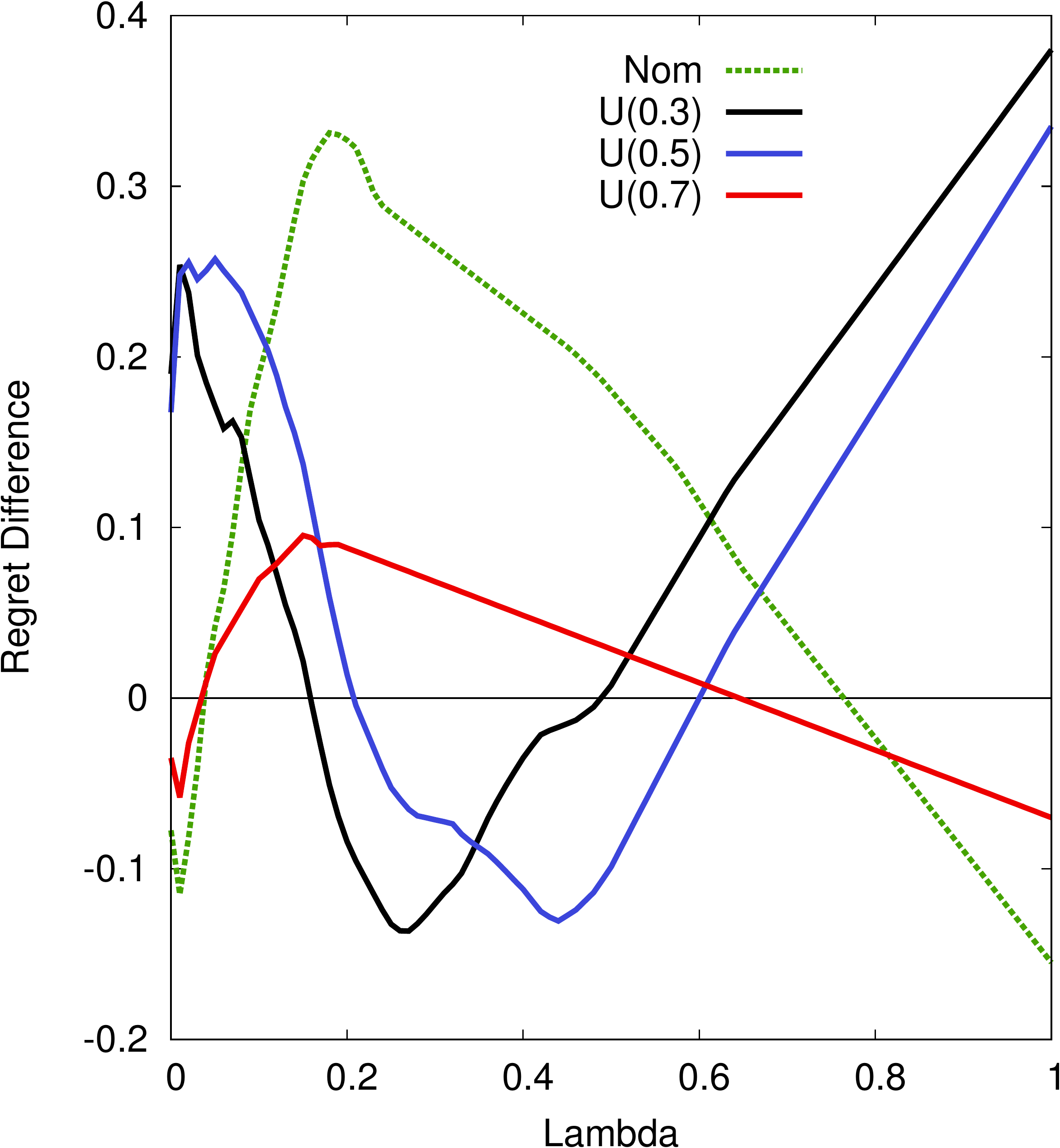}}
\caption{Difference in regret compared to nominal solution depending on $\la$.}\label{fig:exp}
\end{figure}

By construction, a min-max regret solution with respect to $\cU(\bar{\la})$ has the smallest regret for this $\bar{\la}$. Generally, all presented solutions have higher regret than the nominal solution for small and for large values of $\la$, and perform better in between. By construction, the compromise solution has the smallest integral under the shown curve. It can be seen that it presents an interesting alternative to the other solutions by having a relatively small regret for small and large values of $\la$, but also a relatively good performance in between. 

\subsubsection{Two-Path Graphs}

We generate the same plots as in Section~\ref{sec:plots} using the two-path instances. Recall that in this case, the nominal solution is not necessarily an optimal solution with respect to $\cU(1)$. We therefore include an additional line for $\cU(1)$ in Figure~\ref{fig:exp2}.

\begin{figure}[htbp]
   \captionsetup[subfigure]{justification=centering}
    \centering
 \subfloat[][Length = 150, $d=0.05$.\\ Value range {$[197.4, 10608.4]$}.]{\includegraphics[width=.5\textwidth]{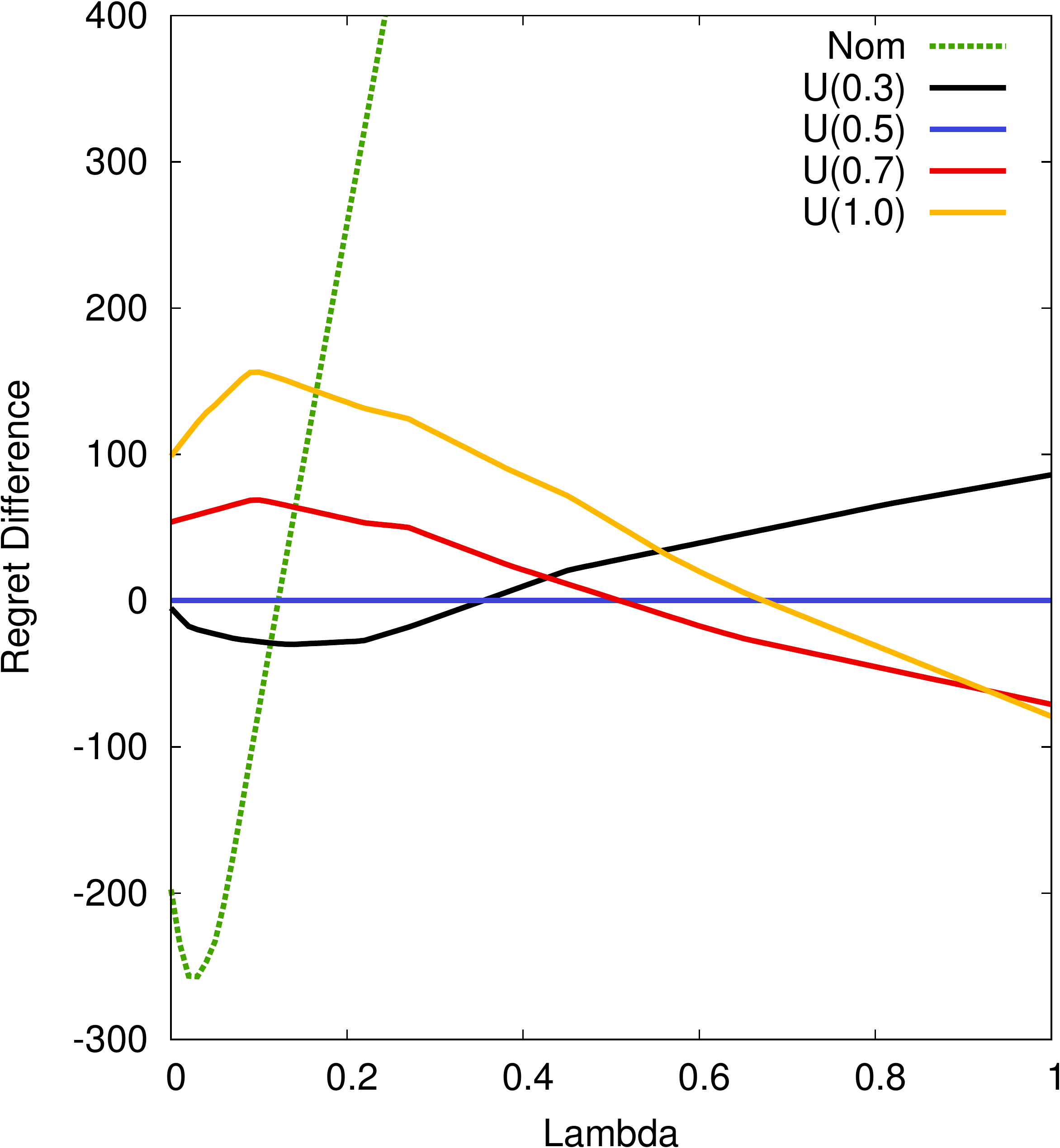}}
 \subfloat[][Length= 150, $d=0.15$.\\ Value range {$[249.1, 11493.1]$}.]{\includegraphics[width=.5\textwidth]{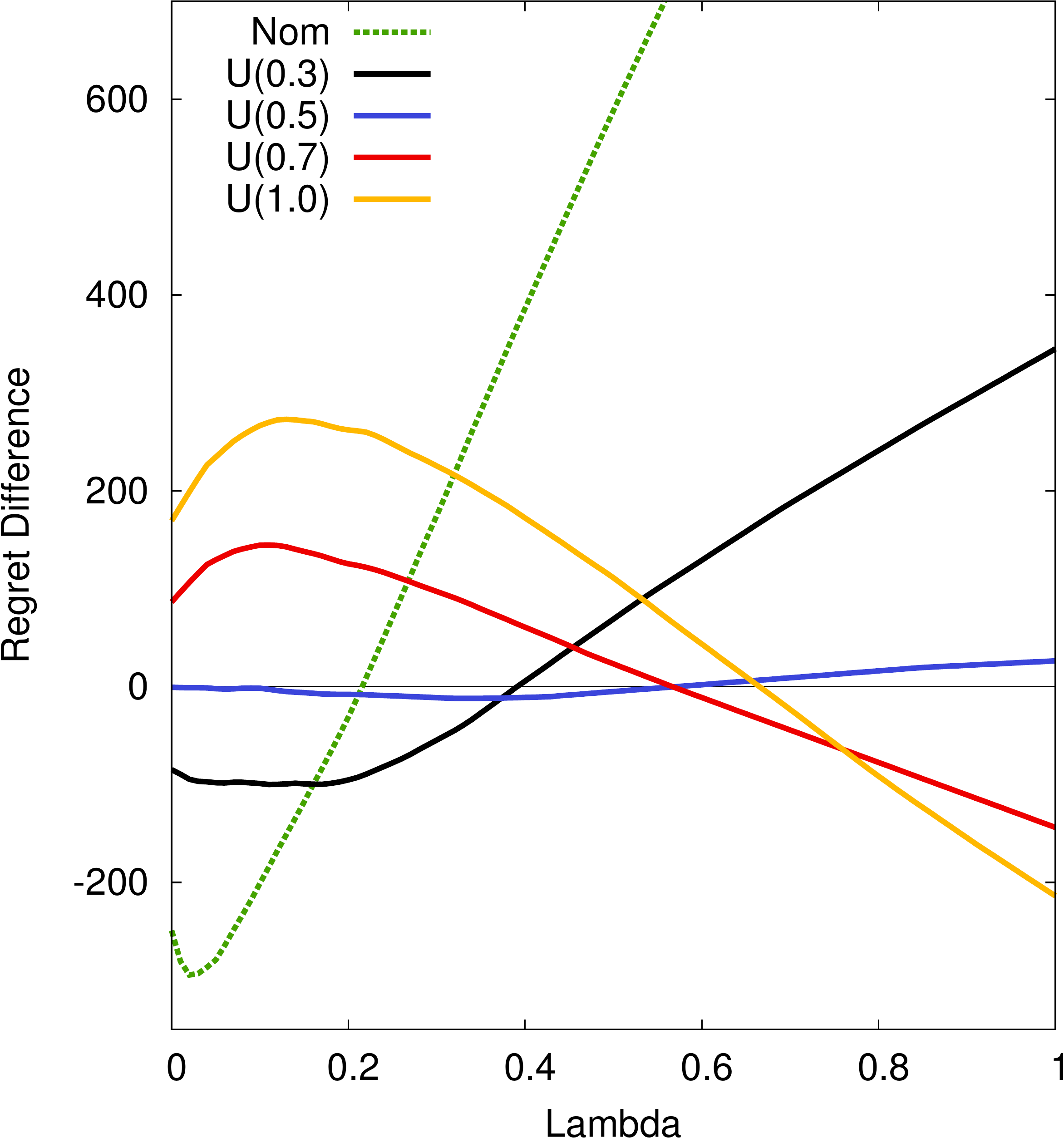}}\\
 \subfloat[][Length = 750, $d=0.05$.\\ Value range {$[1215.2, 51225.2]$}.]{\includegraphics[width=.5\textwidth]{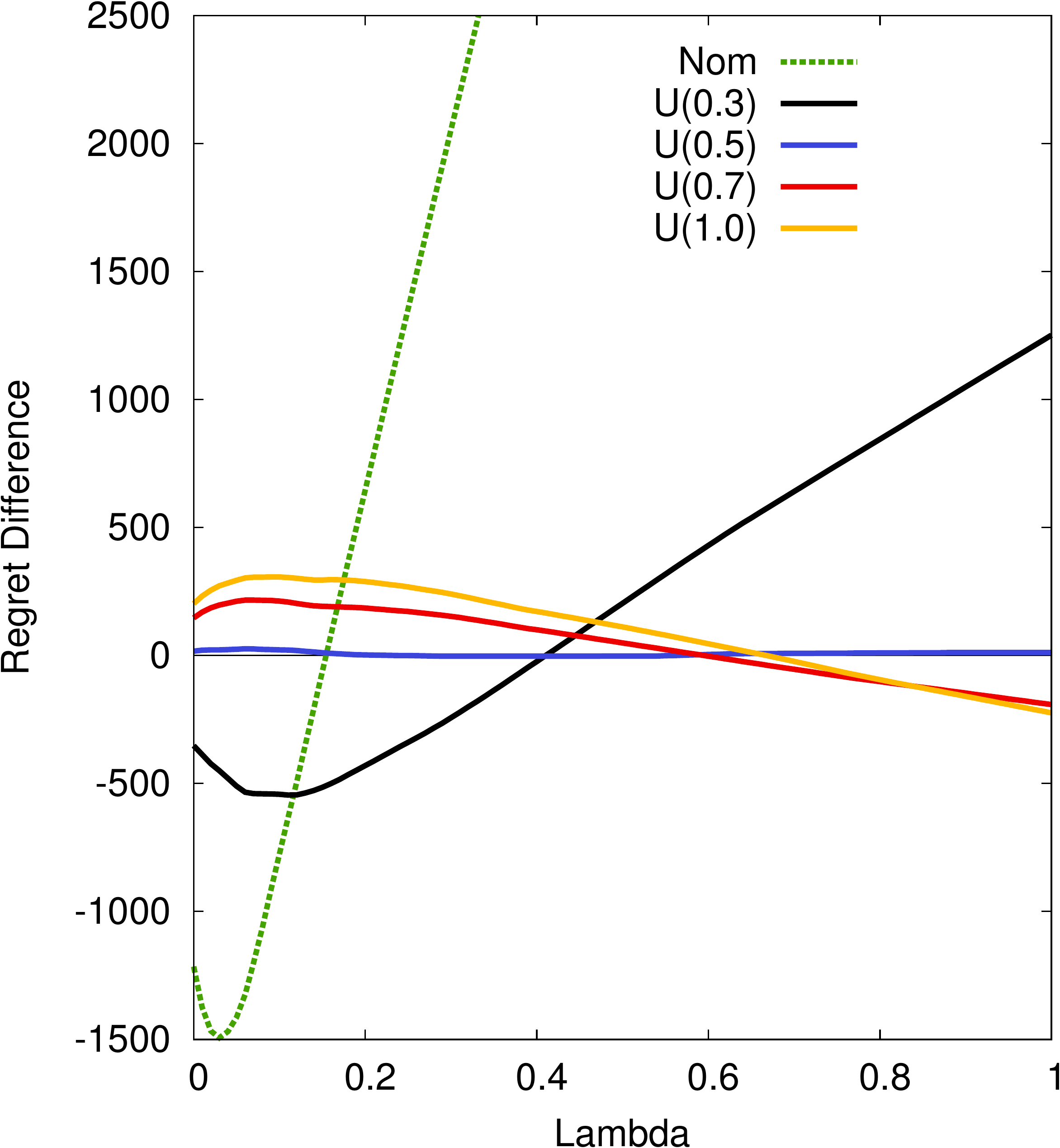}}
 \subfloat[][Length = 750, $d=0.15$.\\ Value range {$[1502.0, 54567.2]$}.]{\includegraphics[width=.5\textwidth]{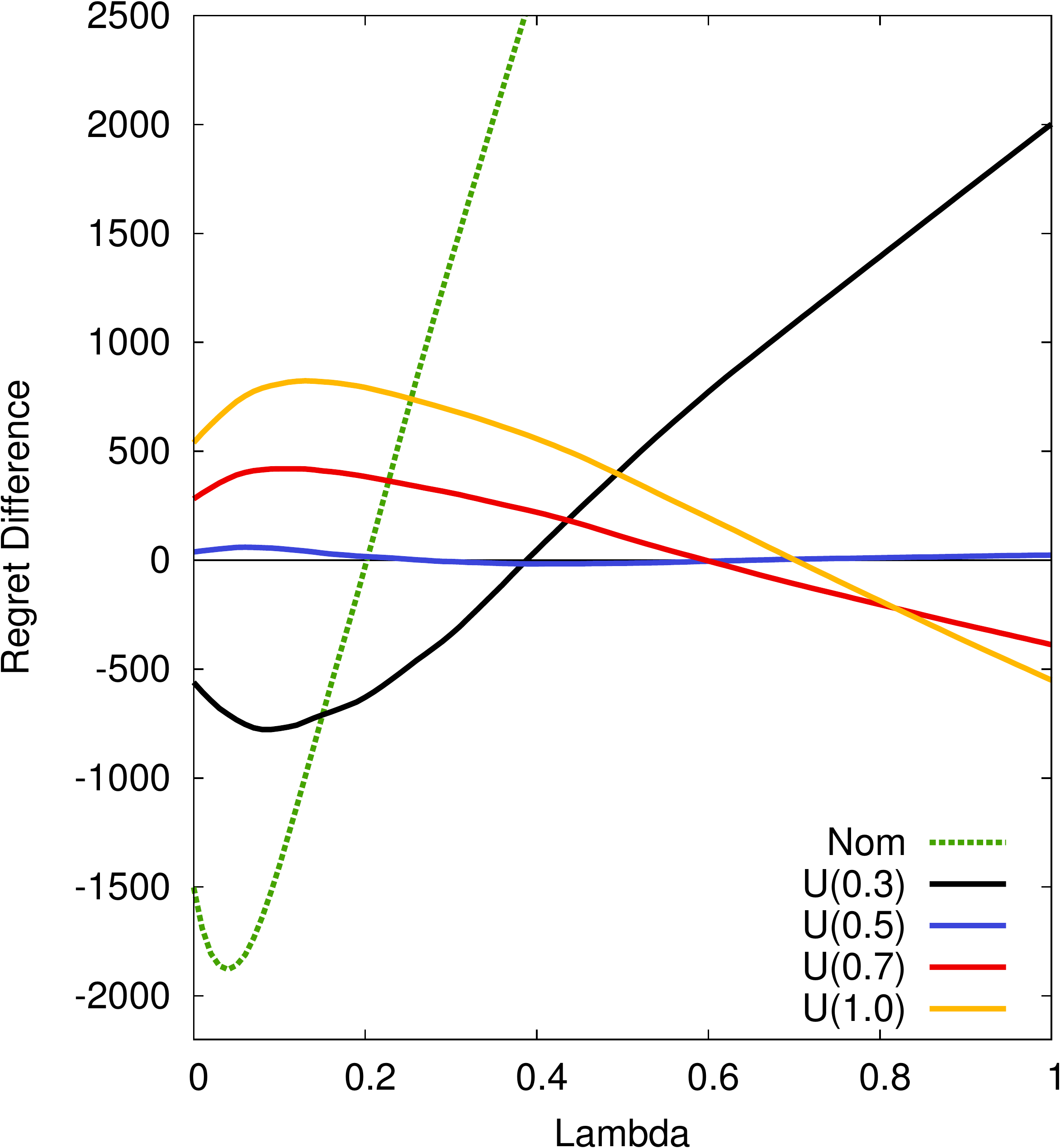}}
\caption{Difference in regret compared to nominal solution depending on $\la$.}\label{fig:exp2}
\end{figure}

It can be seen that the nominal solution performs different to the last experiment; the regret increases with $\lambda$ in a rate that part of the line needed to be cut off from the plot for better readability. The solution to $\cU(0.5)$ performs very close to the compromise solution overall. Additionally, the scale of the plots show that differences in regret are much larger than in the previous experiment. Overall, it can be seen that using a robust solution plays a more significant role than in the previous experiment, as the nominal solution shows poor performance. The solutions that hedge against large uncertainty sets ($\cU(0.7)$ and $\cU(1.0)$) are relatively expensive for small uncertainty sets and vice versa. The compromise solution (as $\cU(0.5)$, in this case) presents a reasonable trade-off over all uncertainty sizes.

\section{Conclusion}\label{sec:conc}

Classic robust optimization approaches assume that the uncertainty set $\cU$ is part of the input, i.e., it is produced using some expert knowledge in a previous step. If the modeler has access to a large set of data, it is possible to follow recently developed data-driven approaches to design a suitable set $\cU$. In our approach, we remove the necessity of defining $\cU$ by using a single nominal scenario, and considering all uncertainty sets generated by deviating coefficients of different size simultaneously. The aim of the compromise approach is to find a single solution that performs well on average in the robust sense over all possible uncertainty set sizes.

For min-max combinatorial problems, we showed that our approach can be reduced to solving a classic robust problem of particular size. The setting is more involved for min-max regret problems, where the regret objective is a piecewise  linear function in the uncertainty size. We presented a general solution algorithm for this problem, which is based on a reduced master problem, and the iterative solution of subproblems of nominal structure.

For specific problems, positive and negative complexity results were demonstrated. The compromise selection problem can be solved in polynomial time. Solutions to the compromise minimum spanning tree problem can be evaluated in polynomial time, but it is NP-hard to find an optimal solution.
For compromise shortest path problems, the same results hold in case of layered graphs; however, for general graphs, it is still an open problem if there exist instances where exponentially many regret solutions are involved in the evaluation problem.

In computational experiments we highlighted the value of our approach in comparison with different min-max regret solutions, and showed that computation times can be within few minutes for instances with up to 22,000 edges.

\end{document}